\documentclass[11pt]{amsart}

\usepackage{amsmath,amssymb,  latexsym, amsthm, amscd, fancyhdr, fullpage}

\usepackage[svgnames]{xcolor}
\usepackage{tikz}

\pgfdeclarelayer{edgelayer}
\pgfdeclarelayer{nodelayer}
\pgfsetlayers{edgelayer,nodelayer,main}

\tikzstyle{none}=[inner sep=0pt]
\definecolor{hexcolor0xf81e1c}{rgb}{0.973,0.118,0.110}
\definecolor{hexcolor0x3c00ff}{rgb}{0.235,0.000,1.000}
\definecolor{hexcolor0x24fe00}{rgb}{0.141,0.996,0.000}

\tikzstyle{whitevertex}=[circle,fill=White,draw=Black, scale = 0.5]
\tikzstyle{vertex}=[circle,fill=White,draw=Black, scale = 0.5]
\tikzstyle{redvertex}=[circle,fill=hexcolor0xf81e1c,draw=Black]
\tikzstyle{bluevertex}=[circle,fill=hexcolor0x3c00ff,draw=Black]
\tikzstyle{greenvertex}=[circle,fill=hexcolor0x24fe00,draw=Black]
\tikzstyle{textbox}=[rectangle,fill=none,draw=none]

\tikzstyle{arc}=[Black, ->]

\newtheorem{theorem}{Theorem}[section]

\newtheorem{corollary}[theorem]{Corollary}
\newtheorem{lemma}[theorem]{Lemma}
\newtheorem{proposition}[theorem]{Proposition}

\begin{document}

\title{Colourings of $(m, n)$-coloured mixed graphs}

\author{Gary MacGillivray, Shahla Nasserasr, Feiran Yang}
\date{}

\maketitle 

\begin{abstract}
A mixed graph is, informally, an object obtained from a simple undirected graph by choosing
an orientation for a subset of its edges.
A mixed graph is $(m, n)$-coloured if each edge is assigned one of $m \geq 0$ colours,
and each arc is assigned one of $n \geq 0$ colours.
Oriented graphs are $(0, 1)$-coloured mixed graphs,
and 2-edge-coloured graphs are $(2, 0)$-coloured mixed graphs.
We show that results of Sopena for vertex colourings of oriented graphs, and
of Kostochka, Sopena and Zhu for vertex colourings oriented graphs and 2-edge-coloured graphs,
are special cases of results about vertex colourings of   
$(m, n)$-coloured mixed graphs.
Both of these can be regarded as a version of Brooks' Theorem.
\end{abstract}

\section{Introduction}

There are parallels between vertex colourings of oriented graphs 
and vertex colourings of 2-edge-coloured graphs:
statements that hold for one family often also hold for the other with more or less  the same proof.
For examples see \cite{AM, BrewsterGraves, KlosMacG, KMRS, KSZ, RS, SenThesis}.
On the other hand, Sen \cite{SenThesis} gives examples where results that hold for oriented graphs do not hold for 2-edge-coloured graphs.  For example, the maximum value of the oriented chromatic number of an orientation of $P_5$ is 3,  while there are 2-edge-colourings of $P_5$ that have chromatic number 4.
Thus it is unlikely that there is a direct translation of results for graphs in one of these families to graphs in the other one.

A connection between oriented graphs and 2-edge-coloured graphs arises through the 
\emph{$(m, n)$-coloured mixed graphs} introduced by Ne\v{s}et\v{r}il and Raspaud \cite{NesetrilRaspaud},
of which both are special cases.
Theorems that hold for $(m, n)$-coloured mixed graphs hold for subfamilies, and
methods which prove such results can be applied to subfamilies.
Conversely, if a statement holds for both 2-edge-coloured graphs and oriented graphs with essentially the same proof, 
then there is some evidence that these may be special cases of a general
theorem for $(m, n)$-coloured mixed graphs.
For example, results in \cite{AM, RS} are shown to hold for $(m, n)$-mixed graphs in \cite{NesetrilRaspaud}.
An ongoing project is to find  generalizations to $(m,n)$-coloured of theorems common to 
2-edge-coloured graphs and oriented graphs.

In this paper we consider vertex colourings of $(m, n)$-coloured mixed graphs and
bounds for the $(m, n)$-coloured mixed chromatic number, $\chi(G, m, n)$.
Definitions and terminology appear in the next section.
In the subsequent sections we extend results which can be considered as
versions of Brooks' Theorem to $(m,n)$-coloured mixed graphs.
Sopena \cite{Sopena} proved constructively  that the oriented chromatic number of an oriented
graph with maximum degree $\Delta \geq 2$ satisfies
$\chi_o \leq (2\Delta - 1)2^{2\Delta - 2}$.
In Section \ref{SecConstructive} we extend this statement 
to $(m, n)$-coloured mixed graphs by using similar methods to prove constructively that
the $(m, n)$-coloured mixed chromatic number of an 
$(m, n)$-coloured mixed graph $G$  with maximum degree $\Delta \geq 2$ 
satisfies $\chi(G, m, n) \leq (2\Delta - 1)(m+2n)^{2\Delta - 2}$.
Kostochka, Sopena and Zhu \cite{KSZ} use the probabilistic method to give a 
better bound.  They show that
the oriented chromatic number of an oriented
graph with maximum degree $\Delta$ satisfies
$\chi_o \leq \Delta^2 2^{\Delta+1}$,
and the corresponding statement holds 
for the number of colours needed for a vertex-colouring of a 2-edge-coloured graph.
In Section \ref{SecProb} 
we extend this statement 
to $(m, n)$-coloured mixed graphs by using similar methods to show that
$\chi(G, m, n) \leq \Delta^2(m+2n)^{\Delta + 1}$.

Das, Nandi and Sen \cite{DNS} proved that 
$\chi(G, m, n) \leq 2(\Delta - 1)^{m+2n}(m + 2n)^{\Delta-\min\{m+2n, 3\}+2}$ 
for $m+2n \geq 2$ and $\Delta \geq 5$, which is better
than either bound above when $m+2n = 2$, hence in particular for oriented graphs.
For each fixed value of $\Delta \geq 5$, the bound of $\Delta^2(m+2n)^{\Delta + 1}$ is better
for sufficiently large values of $m+2n$.
Das, Nandi and Sen also showed that there exist $(m, n)$-coloured mixed graphs $G$ with 
$\chi(G, m, n)$ $ \geq (m+2n)^{\Delta/2}$, and established connections between 
the $(m, n)$-coloured mixed chromatic number, the acyclic chromatic number of the underlying graph, and the arboricity of the underlying graph (also see \cite{NesetrilRaspaud}).

The $(m, n)$-coloured mixed chromatic number of planar graphs, partial 2-trees
and outerplanar graphs with given girth is studied in \cite{MPRS}.
A survey of results about the $(m, n)$-coloured mixed chromatic number of 
other families of $(m, n)$-coloured mixed graphs is given in 
\cite{DNRS}.
Further results on colourings and homomorphisms of $(m,n)$-coloured mixed graphs appear in
\cite{DuffyThesis}.

\section{$(m,n)$-coloured mixed graphs}

A \textit{mixed graph} is an ordered triple $G=(V(G), E(G), A(G))$, where 
$V(G)$ is a set of objects called \emph{vertices},
$E(G)$ is a set of unordered pairs of distinct vertices called \emph{edges}, 
$A(G)$ is a set of ordered pairs of distinct vertices called \emph{arcs},
and the underlying undirected graph of $G$ is a simple graph.
Any two adjacent vertices in a mixed graph are joined by a single edge or a single arc, and not both.

It is possible to define mixed graphs more generally so that loops and various types of
multiple adjacencies are allowed.
The objects defined in the previous paragraph would then be ``simple'' mixed graphs.
The notion of vertex colouring we consider makes sense only for ``simple''
mixed graphs, so we have chosen to omit the more general definition. 
We will not use the adjective ``simple'' in the remainder of the paper.

A mixed graph $G$ is \emph{$(m, n)$-coloured} if each edge is assigned one of the 
$m$ colours $1, 2, \ldots, m$ and each arc is assigned one of the $n$ colours $1, 2, \ldots, n$.
For $1 \leq i \leq m$, let $E_i(G)$ be \emph{the set of edges of colour $i$}, and
for $1 \leq j \leq n$, let  $A_j(G)$ be  \emph{the set of arcs of colour $j$}.

When the context is clear we write $V, E$ and $A$ instead of $V(G), E(G)$ and $A(G)$, respectively, and 
similarly for other subsets or parameters related to a graph.

We sometimes say that an $(m, n)$-coloured mixed graph has a property belonging to
its underlying undirected graph.
We call an $(m, n)$-coloured mixed graph \emph{complete} if its underlying undirected graph is complete.  
The  \emph{maximum degree}, $\Delta$, of an $(m, n)$-coloured mixed graph is the
maximum degree of its underlying undirected graph.  

Let $G$ and $H$ be $(m, n)$-coloured mixed graphs.
A ($(m,n)$-coloured) \emph{homomorphism} of $G$ to $H$ is a function $f: V(G) \to V(H)$ such that
if the edge $ab\in E_i(G)$, then $f(a)f(b)\in E_i(H)$,
and if the arc $ab\in A_j(G)$ then $f(a)f(b)\in A_j(H)$. 
A homomorphism of  $G$ to $H$ preserves edges, arcs, and colours.
We may write $G\rightarrow H$ to denote the existence of 
a homomorphism of $G$ to $H$.

A (vertex) $k$-\emph{colouring} of an $(m, n)$-coloured mixed graph $G$ is a homomorphism to an $(m, n)$-coloured mixed graph on $k$ vertices.
As in \cite{NesetrilRaspaud}, we define  the smallest integer $k$ such that there exists a $k$-colouring of $G$ to be the
\emph{$(m, n)$-coloured mixed chromatic number of $G$}, and denote it by
$\chi(G, m, n)$.
Note that, for graphs (i.e.,  $(1, 0)$-coloured mixed graphs), the quantity $\chi(G, 1, 0)$ is the usual chromatic number, 
and for oriented graphs (i.e.,  $(0, 1)$-coloured mixed graphs),  the quantity $\chi(G, 0, 1)$ is the oriented chromatic number.

Let $k$ be a positive integer.
An $(m, n)$-coloured mixed graph $H$ is called \emph{$k$-hom-universal} if $G \to H$ 
whenever $G$ is an $(m, n)$-coloured mixed graph with $\Delta \leq k$.
The number of vertices in a $k$-hom-universal $(m,n)$-coloured mixed graph $H$ is an upper bound 
on the $(m, n)$-coloured mixed chromatic number of any 
$(m, n)$-coloured mixed graph with $\Delta \leq k$.

Let $G$ be an $(m, n)$-coloured mixed graph, and let $x \in V$.
If $y$ is adjacent to $x$, then 
we record the colour and orientation of the edge or arc joining $x$ and $y$ (with respect to $x$)
by an integer  $c \in \{1, 2, \ldots, m+2n\}$:  
\begin{itemize}
\item if $c \in \{1, 2, \ldots, m\}$ then $x$ and $y$ are joined by an edge of colour $c$;
\item if $c \in \{m+1, m+2, \ldots, m+n\}$ then there is an arc of colour $c-m$ from $x$ to $y$;
\item if $c \in \{m+n+1, m+n+2, \ldots, m+2n\}$ then there is an arc of colour $c-(m+n)$ from $y$ to $x$.
\end{itemize}

Let $X  = \{v_1, v_2, \ldots, v_d\}$ be a subset of vertices of an $(m,n)$-coloured mixed graph $G$, and let $x$ be a vertex of $G$ which is adjacent to every vertex in $X$.  Denote by $a_G(x, X)$ the element of $\mathbb{Z}_{m+2n}^d$ whose $j$-th coordinate 
records the colour and orientation of $xv_j$.

Let $t \geq 0$ be an integer.  
An $(m, n)$-coloured mixed graph $G$ is said to have 
\emph{Property $P_{a, b}$} if, for every $(m, n)$-coloured complete subgraph with vertex set 
$X = \{v_1, v_2, \ldots, v_\ell\}$, where $0 \leq \ell \leq a$, and 
every $\ell$-tuple $L = (c_1, c_2, \ldots, c_\ell) \in \mathbb{Z}_{m+2n}^\ell$, there exist
at least $b$ vertices $x \in V(G)$ such that 
$a_G(x, X) = L$.
Informally, $G$ has Property $P_{a, b}$ if, for every $(m, n)$-coloured complete subgraph of size at most $a$ and every possible way of extending it to a complete subgraph of size one greater, there are at least $b$ vertices with the required adjacencies.

\section{A constructive Brooks' Theorem}
\label{SecConstructive}

In this section we generalize results of Sopena \cite{Sopena} for oriented graphs to $(m,n)$-coloured mixed graphs by showing that, for each $k \geq 2$, each member of a family 
of  $(m, n)$-coloured mixed graphs defined by 
Ne\v{s}et\v{r}il and Raspaud \cite{NesetrilRaspaud} is $k$-hom-universal.
This family of $(m, n)$-coloured mixed graphs is defined in two steps.
We first define a family of $(m, n)$-coloured mixed bipartite graphs, each of which
gives rise to a different member of the family we will ultimately use.

Let $\mathcal{H}_{m+2n, m+2n}$ be the family of $(m, n)$-coloured mixed bipartite graphs obtained from 
the complete bipartite graph $K_{m+2n, m+2n}$ with bipartition $(A, B)$ as follows.
Let $A = \{a_1, a_2, \ldots, a_{m+2n}\}$ and $B = \{b_1, b_2, \ldots, b_{m+2n}\}$.
Let $F_1, F_2, \ldots, F_{m+2n}$ be a 1-factorization of  $K_{m+2n, m+2n}$:
\begin{itemize}
\item for $i = 1, 2, \ldots, m$, colour each edge in $F_i$ with colour $i$;
\item for $j = 1, 2, \ldots, n$, replace each edge in $F_{m+j}$ by an arc of colour $j$ oriented from 
$A$ to $B$, and each edge in  $F_{m+n+ j}$ by an arc of colour $j$ oriented from $B$ to $A$.
\end{itemize}
The underlying undirected graph of each $H \in \mathcal{H}_{m+2n, m+2n}$ is $K_{m+2n, m+2n}$.
If $H \in \mathcal{H}_{m+2n, m+2n}$, then every vertex of $H$ is incident with an edge of each colour,
and an arc of each colour and each possible orientation.

The family of $(m,n)$-coloured mixed graphs $\mathcal{Z}_{m,n,q}$, $q \geq 1$, is then defined as follows.
Each member of the family has vertex set 
$$V = \{(i; v_1, v_2, \ldots, v_q): 1 \leq i \leq q,\ v_i = ``\cdot", \mbox{ and } 1 \leq v_j \leq m + 2n \mbox{ for } 1 \leq j \leq q\}.$$
A vertex $(i; v_1, v_2, \ldots, v_q)$ is said to have \emph{index $i$}.
To define the edge set of $Z \in \mathcal{Z}_{m,n,q}$, first choose $H \in \mathcal{H}_{m+2n, m+2n}$. 
(Different choices of $H$ lead to different $(m,n)$-coloured mixed graphs $Z$.)
Let ${\bf v} = (i; v_1, v_2, \ldots, v_q)$ and ${\bf w} = (j; w_1, w_2, \ldots, w_q)$ be vertices of $Z$,
where $i < j$.
Let $s = v_j$ and $t = w_i$.
The vertices ${\bf v}$ and ${\bf w}$ are joined by
\begin{itemize}
\item an edge of colour $c$ if  $a_s \in A$ and $b_t \in B$ are joined in $H$ by an edge of colour $c$;
\item an arc of colour $c$ oriented from ${\bf v}$ to ${\bf w}$ if $a_s \in A$ and $b_t \in B$ are joined in $H$ by an arc of colour $c$ oriented from $a_s$ to $b_t$;
\item an arc of colour $c$ oriented from ${\bf w}$ to ${\bf v}$ if $a_s \in A$ and $b_t \in B$ are joined in $H$ by an arc of colour $c$ oriented from $b_t$ to $a_s$.
\end{itemize}
The underlying undirected graph of each $Z \in \mathcal{Z}_{m,n,q}$ is the complete $q$-partite graph in which each class of the partition has size $m + 2n$, and contains $(m+2n)^{q-1}$ vertices.



\begin{proposition}
Let $q \geq 2$ be an integer.
If $Z \in \mathcal{Z}_{m, n, q}$, then $Z$ has Property $P_{q-1, 1}$.
\label{PropExtend}
\end{proposition}

\begin{proof}
This follows immediately from the construction of $Z$.
\end{proof}

\begin{theorem}\label{Brooksish}
Let $k \geq 2$ be an integer.
Each $(m, n)$-coloured mixed graph $Z \in \mathcal{Z}_{m,n,2k-1}$
is $k$-hom-universal.
\label{ThmMapping}
\end{theorem}

\begin{proof}
We must show that, for any  $(m, n)$-coloured mixed graph $G$ with $\Delta \leq k$
and any $Z \in \mathcal{Z}_{m,n,2k-1}$,
there is a homomorphism $G \to Z$.
We will prove by induction on $p = |V(G)|$ the stronger statement that if $G$ is a $(m, n)$-coloured mixed graph with $\Delta \leq k$ and any $Z \in \mathcal{Z}_{m,n,2k-1}$, there is a homomorphism $G \to Z$ such that for every $v \in V(G)$ the images of neighbours of $v$ have distinct indices.
%
The statement is clear if $p=1$.  
Suppose it holds for all $(m, n)$-coloured mixed graphs on at most $p-1$ vertices, for some $p \geq 2$.
Consider an $(m, n)$-coloured mixed graph $G$ on $p$ vertices.

Let $x \in V(G)$.
Let the neighbours of $x$ in the underlying undirected graph of $G$ 
be $v_1, v_2, \ldots, v_j$, where $j \leq \Delta \leq k$.

By the induction hypothesis, 
there is a homomorphism $(G-x) \to Z$ such that 
for every $v \in V(G-x)$ the images of neighbours of $v$ have distinct indices.

We claim that in any such homomorphism there are at least $k$ possible images for $v_1$.
Since $v_1$
has degree at most $\Delta-1 \leq k-1$ in $G-x$, and
the images of its neighbours have different indices, at most $k-1$ indices
are forbidden for the image of $v_1$ 
(adjacent vertices can not be assigned images with the same index).
Note, also, that each neighbour of $v_1$ determines a different component in 
the image of $v_1$.  
Therefore, there are at least $(2k-1)-(k-1)=k$ possible images for $v_1$.
This proves the claim.

Similarly, in any such homomorphism there are at least $k$ possible images for each of $v_2, v_3, \ldots, v_j$.
Thus, since $j \leq k$, there exists a homomorphism
of $G-x$ to $Z$ such that 
that the images of the neighbours of $x$ (in $G$)
also have distinct indices.

Since $2k-2 \geq k$ and $Z$ has Property $P_{2k-2, 1}$ 
(by Proposition \ref{PropExtend}), the
mapping can be extended to $x$.
This completes the proof.
\end{proof}

\begin{corollary}
Let $G$ be an $(m, n)$-coloured mixed graph with $\Delta \geq 2$.
Then $\chi(G, m, n) \leq (2\Delta-1)(m+2n)^{2\Delta-2}$.
\label{CorSopena}
\end{corollary}

\begin{proof}
By Theorem  \ref{ThmMapping}, the $(m,n)$-coloured mixed graph $G$ has a
homomorphism to any $Z \in \mathcal{Z}_{m,n,2\Delta-1}$.
The number of vertices of $Z$ is $(2\Delta-1)(m+2n)^{2\Delta-2}$.
\end{proof}

The above corollary generalizes Theorem 4.1 of Sopena \cite{Sopena}.
For oriented graphs, when $\Delta = 1$ the quantity 
$(2\Delta-1)(m+2n)^{2\Delta-2} = 1$, 
while the oriented chromatic number of an oriented graph with maximum 
degree 1 equals 2.
For connected $(m,n)$-coloured graphs with $\Delta = 1$ this gives
$\chi(G, m, n) = 2$.
The $(m,n)$-coloured mixed chromatic number of an 
arbitrary $(m,n)$-coloured mixed graph with $\Delta = 1$
can be larger than 2: consider the the disjoint union of $m$ edges and $n$ arcs, 
one of each possible colour.  
The minimum number of vertices in a 1-hom-universal $(m,n)$-coloured mixed graph is 
the smallest positive integer $c$ such that $\binom{c}{2} \geq m+n$.  

\section{A probabilistic Brooks' Theorem}
\label{SecProb}

Kostochka, Sopena and Zhu \cite{KSZ} used the probabilistic method to prove
the existence of a $k$-hom-universal graph for $k \geq 4$, and  
improve bound for the oriented chromatic number given in \cite{Sopena} as a consequence.  
They note that ``the same'' proof shows that the same bound holds for $\chi(G, 2, 0)$.
In this section we extend their result to $(m,n)$-coloured mixed graphs.


\begin{lemma}
For all  integers $k \geq 4$ and $m,n$ with $m+2n \geq 3$, there exists an $(m,n)$-coloured mixed graph $H$  on $t = k^2(m+2n)^{k+1}$ vertices with Property $P_{i, (k-i)(k-1) + 1}$ for $i = 1, 2, \ldots, k$.
\label{target}
\end{lemma}
\begin{proof}
Let $c=m+2n$. We shall show that the probability that a random $(m,n)$-coloured mixed complete graph $H$ on $t$ vertices, in which every two vertices are connected by either an edge of some colour or an arc of some colour and orientation independently with probability $1/c$, has the above property with positive probability.


Let $1 \leq i \leq k$ and $X$ be an $i$-subset of $V(H)$.  Let $b \in \mathbb{Z}_{c}^i$, and $\mathcal{E}_{X, b}$ be the event that the number of vertices in $x \in V(H - X)$ with  $a_H(x, X) = b$ is at most $(k-i)(k-1)$, that is, that $H$ does not have Property $P_{i, (k-i)(k-1) + 1}$.  
We want to show that the probability of $\mathcal{E}_{X, b}$ is less than one.

For any fixed vertex $x \in V(H-X)$, we have $[\mathit{Pr}(a_H(x, X) = b)] = c^{-i}$.  Since, for different vertices $x, y \in V - X$ the events $a_H(x, X) = b$ and $a_H(y, X) = b$ are independent, we have
\begin{eqnarray*}
\mathit{Pr}(\mathcal{E}_{X, b}) & = & \sum_{j=0}^{(k-i)(k-1)} {t-i \choose j} (c^{-i})^j (1 - c^{-i})^{((t-i)-j)}\cr
& \leq & (1-c^{-i})^t  \sum_{j=0}^{(k-i)(k-1)} \frac{t^j}{j!} m^{-ij }(1 - c^{-i})^{-(i+j)}\cr
& < & \left[\frac{c-1}{c-2}\right] e^{-tc^{-i}}  \sum_{j=0}^{(k-i)(k-1)} t^j\cr
& < & e^{-tc^{-i}}   t^{(k-i)(k-1)+1}\cr
\end{eqnarray*}
The probability that some event  $\mathcal{E}_{X, b}$ occurs satisfies the following
\begin{eqnarray*}
\mathit{Pr}\left(\bigcup_{X, b} \mathcal{E}_{X, b}\right)& = & \sum_{i=1}^k \ \sum_{|X| = i}\  \sum_{b \in \mathbb{Z}_{c}^i} \mathit{Pr}(a_T(x, X) = b) \cr
& < & \sum_{i = 0}^k {t \choose i} c^i  e^{-tc^{-i}}   t^{(k-i)(k-1)+1}\cr
& < & c^k \sum_{i = 0}^k   e^{-tc^{-i}}   t^{[(k-i)(k-1)+1]+i}\cr
& < &c^k \sum_{i = 0}^k   e^{-tc^{-i}}   t^{(k-i)(k-1)+1+i}\cr
\end{eqnarray*}
In the last sum, for $i=0,1,\ldots, k-1$, the ratio of the $(i+1)$-st term to the $i$-th term is
$$\frac{e^{tc^{-i}} t^{(k-(i+1))(k-1)+1 + (i + 1)}}{e^{tc^{-(i+1)}}  t^{(k-i)(k-1)+1 + i}} 
= \frac{e^{t(c-1)c^{-(i+1)}}}{t^{k-2}}
\geq \frac{e^{t(c-1)c^{-k}}}{t^{k-2}}.
$$
When $t = k^2c^{k+1}$ we have
$$
\frac{e^{t(c-1)cc^{-k}}}{t^{k-2}}
= \frac{e^{k^2(c-1)c}}{(k^2c^{k+1})^{k-2}}
= \frac{e^{k^2(c-1)c}}{k^{2k-4}c^{(k+1)(k-2)}}.
$$
We show the ratio is greater than $c$, that is
$$e^{k^2(c-1)c} > ck^{2k-4}c^{(k+1)(k-2)}.$$ On taking logs of both sides, we see that this is true if and only if $k^2(c-1)c > (2k-4)\ln(k) + [(k+1)(k-2) + 1]\ln(c)$, which holds for all integers $c$ and $k$ with $k \geq 4$ and $c \geq 3$.

Hence, 
\begin{eqnarray*}
\mathit{Pr}\left(\bigcup_{X, b} \mathcal{E}_{X, b}\right)& < & c^k \sum_{i = 0}^k   e^{-tc^{-i}}   t^{(k-i)(k-1)+1+i}\cr
& < & (c / (c-1))\, c^k e^{-tc^{-k}}t^{k+1}\cr
& = & (c / (c-1))\, c^k e^{-ck^2}(k^2c^{k+1})^{k+1}\cr
& = & (c / (c-1))\,  e^{-ck^2} k^{2+2k} c^{(k+1)^2+k}.\cr
\end{eqnarray*}
In order for the probability to be less than one, it is required that
$ k^{2+2k}c^{(k+1)^2+k+1}\leq (c-1) e^{ck^2}$, or equivalently (after taking logs) that 
$$(2k+2)\ln k +(k+1)(k+2)\ln c < \ln(c-1) +ck^2$$
which holds for $k \geq 4$ and $c \geq 3$. This completes the proof.
\end{proof}


\begin{theorem}
Let $G$ be an $(m, n)$-coloured mixed graph.
Then $\chi(G, m, n) \leq \Delta^2 (m+2n)^{\Delta+1}$.
\end{theorem}
\begin{proof}
The statement is clear if $\Delta = 1$.
For $2 \leq \Delta \leq 3$, the statement follows from Corollary \ref{CorSopena}, so assume $\Delta \geq 4$.
Let $H$ be the $(m, n)$-coloured graph whose existence is asserted by Lemma \ref{target}.

Suppose $V(G) = \{v_1, v_2, \ldots, v_n\}$.
For $t = 1, 2, \ldots, n$, let $G_t$ be the subgraph of $G$
induced by $\{v_1, v_2, \ldots, v_t\}$.
Define a homomorphism $f_n: G \to H$ inductively, as follows.
Suppose $f_t: G_t \to H$ has the property that if $v_j, v_k \in V(G_t)$ have
a common neighbour in the underlying undirected graph of $G$,
then $f_t(v_j) \neq f_t(v_k)$.

Suppose $v_{t+1}$ is adjacent in the underlying
undirected graph of $G$ to 
$w_1, w_2, \ldots, w_i \in V(G_t)$.
Let $W_{t+1} = \{w_1, w_2,$ $\ldots, w_i\}$, $a_H(v_{t+1}, W_{t+1}) = b$, and
$X_{t+1} = \{x \in V(H) - W_{t+1}: a_H(x , W_{t+1}) = b\}$.

By Lemma \ref{target}, $|X_{t+1}| \geq 1 + (\Delta - i)(\Delta - 1)$.
Let $Y_{t+1}$ be the set of vertices in $\{v_{t+1}, v_{t+2}, \ldots,$ $v_n\}$
which are adjacent in $G$ to $v_{t+1}$.
Then $|Y_{t+1}| \leq \Delta - i$.
Let $Z_{t+1}$ be the set of vertices $G_{t+1}$ which are adjacent to a vertex in $Y_{t+1}$.
Then $|Z_{t+1}| \leq (\Delta - 1)(\Delta - i)$.
Thus there is a vertex $z_{t+1} \in X_{t+1} - f_t(Z_{t+1})$ 
(recall that vertices with a common neighbour have different images under $f_t$).
Extend $f_t$ to a homomorphism $f_{t+1}: G_{t+1} \to H$ by setting
$f_{t+1}(v_{t+1}) = z_{t+1}$. 
The function $f_{t+1}$ has the desired property.
The result follows.
\end{proof}

\vfill

\noindent
\textsc{Gary MacGillivray
\footnote{Research supported by NSERC}, Mathematics and Statistics, University of Victoria, Victoria, Canada}, 
{\tt gmacgill@math.uvic.ca}\\

\smallskip\noindent
\textsc{Shahla Nasserasr $^1$, School of Mathematical Sciences, Rochester Institute of Technology, Rochester, NY, USA}, 
{\tt shahla@mail.rit.edu}\\

\smallskip\noindent
\textsc{Feiran Yang, Department of Mathematics and Statistics, University of Victoria, Victoria, Canada}, 
{\tt fyang@math.uvic.ca}\\

\bigskip\noindent
{\bf Acknowledgement}.  
The authors are indebted to the referee for pointing out that in the proof of Theorem \ref{Brooksish} it is necessary to make the stronger induction hypothesis that the neighbours of each vertex have distinct images.  Without that, the proof fails.  They also pointed out that the same omission occurs in the proof of the corresponding theorem in \cite{Sopena}.

\end{document}